\documentclass[10pt,twoside]{article}
\usepackage{mathrsfs}
\usepackage{amssymb}
\usepackage{amsmath,color}
\usepackage[all]{xy}
\usepackage{amsthm}
\usepackage{url}
\usepackage{setspace}
\usepackage{hyperref}
\hypersetup{colorlinks,linkcolor=blue,citecolor=blue}     %%% 为引用插入超链接

\usepackage{tikz}
\usepackage{stmaryrd}
\usepackage{xcolor}
\usetikzlibrary{arrows,calc}
\usepackage{etex}

\numberwithin{equation}{section}

 \def\Hom{\mbox{\rm Hom}}

\def\Mod{\mbox{\rm Mod}\,}

\def\cone{\mbox{\rm cone}}
\def\Im{\mbox{\rm Im}}

\def\P{\mathcal {P}}
\def\A{\mathcal{A}} 
\def\Id{\mbox{\rm Id}\,} \def\Im{\mbox{\rm Im}\,}

\def\E{\mathbb{E}}

\def\del{\delta}
\def\B{\mathcal {B}}\def\C{\mathcal {C}}
\def\X{\mathcal{X}}
\def\Y{\mathcal{Y}}
\def\Z{\mathcal{Z}}
\def\I{\mathcal{I}}

\newcommand{\res}{\operatorname{res.dim}}
\newcommand{\cores}{\operatorname{cores.dim}}

\newcommand{\bsm}{\begin{smallmatrix}}
\newcommand{\esm}{\end{smallmatrix}} %%%%%%\bsm

\usepackage{geometry}
\usepackage{paralist}

\newtheorem{theorem}{Theorem}[section]
\newtheorem{proposition}[theorem]{Proposition}
\newtheorem{definition}[theorem]{Definition}
\newtheorem{remark}[theorem]{Remark}
\newtheorem{lemma}[theorem]{Lemma}

\newtheorem{corollary}[theorem]{Corollary}
\newtheorem{condition}[theorem]{Condition}
\newtheorem{theorem*}{Theorem}

\newcommand{\cotd}{\operatorname{cot.dim}}

\newcommand{\Ker}{\operatorname{Ker}}

\newcommand{\RNum}[1]{\uppercase\expandafter{\romannumeral #1\relax}}% 罗马数字

\title{ \bf Dimensions and cotorsion pairs in recollements of extriangulated categories
\footnotetext{Xin Ma was supported by the National Natural Science Foundation of China (12001168), Henan University of Engineering (DKJ2019010). Panyue Zhou was supported by the National Natural Science Foundation of China (Grant No. 12371034) and by the Hunan Provincial Natural Science Foundation of China (Grant No. 2023JJ30008).} }
\author{Xin Ma and Panyue Zhou}
\date{ }
\begin{document}

\baselineskip=16pt
\maketitle

\begin{abstract}
\begin{spacing}{1.3}
Let $(\A, \B, \C)$ be a recollement of extriangulated categories. In this paper, we provide bounds on the coresolution dimensions of the subcategories involved in $\A$, $\B$ and $\C$. We show that a hereditary cotorsion pair in $\B$ can induce hereditary cotorsion pairs in $\A$ and $\C$ under certain conditions. Besides, we construct recollements of abelian categories from a recollement of extriangulated categories with respect to the hearts of cotorsion pairs.\end{spacing}
\vspace{2mm}

\hspace{-5mm}\textbf{Keywords:} extriangulated category; recollement; coresolution dimension; cotorsion pair\\[0.2cm]
\end{abstract}

\pagestyle{myheadings}
\markboth{\rightline {\scriptsize X. Ma, P. Zhou\hspace{2mm}}}%%%%%%%%%%%    %%%%%%%%    %%%%%%%%%    %%%%%%% 页上角
{\leftline{\scriptsize Dimensions and cotorsion pairs in recollements of extriangulated categories}}%%%%%%%   %%%%%    %%%%%%    %%%%%%%页上角

%\section{Introduction} %delete * to number this section

%\tableofcontents

\section{Introduction}
The concept of an extriangulated category was first introduced by Nakaoka and Palu \cite{Na}, offering a comprehensive generalization that encompasses both exact categories and triangulated categories. Moreover, there exist various examples of extriangulated categories that do not fall within the categories of exact or triangulated (refer to \cite{HZZ20P, INY18A, Na,ZZ18T} for further details). Many results from exact categories and triangulated categories can be unified within the same framework
(see \cite{GNP21, HZZ21G, HZZ20P, HZZ21P, HZ21R, INY18A, Liu, ZZ20T, ZZ21G}, which is used to extract properties related to both exact and triangulated categories.

Recently, Wang, Wei, and Zhang introduced the concept of recollements of extriangulated categories in their work \cite{WWZ20R}. This notion provides a unified generalization of recollements in both abelian and triangulated categories, as originally introduced by Beilinson, Bernstein, and Deligne \cite{BBD}. It's worth noting that recollements in abelian categories and triangulated categories are inherently intertwined. For instance, Chen \cite{C13C} demonstrated the construction of recollements in abelian categories based on recollements in triangulated categories. Furthermore, researchers have delved into gluing techniques with respect to recollements of extriangulated categories. For example, in the context of a recollement of extriangulated categories, He, Hu, and Zhou \cite{HHZ21T} explored gluing torsion pairs. Ma, Zhao, and Zhuang \cite{MZZ} delved into gluing resolving subcategories, investigating resolving resolution dimensions, while Gu, Ma, and Tan \cite{GMT} delved into the study of global dimensions and extension dimensions, among other topics.
These findings collectively represent a comprehensive extension of the classical notions of recollements in both abelian and triangulated categories.

Cotorsion pairs provide a perspective on resolutions within a given category, and they are intricately connected to other homological concepts of significant interest in category theory and representation theory. Additionally, cotorsion dimensions play a crucial role in the study of characterizations of specific rings. Nakaoka and Palu's groundbreaking work \cite{Na} (also explored in \cite{Liu}) introduced the notion of cotorsion pairs in extriangulated categories, offering a simultaneous generalization of cotorsion pairs found in exact categories \cite{L13H} and triangulated categories \cite{N13G}. Moreover, Liu and Nakaoka's research \cite{Liu} proved that the heart of a cotorsion pair in extriangulated categories is abelian, providing a unified approach suitable for both exact and triangulated categories. In the context of a recollement of abelian categories, Fu and Hu \cite{FH22T} delved into the study of cotorsion triples.

In this paper, our focus lies on the exploration of coresolution dimensions, especially cotorsion dimensions, and cotorsion pairs within the framework of a recollement of extriangulated categories. In Section 2, we offer a summary of fundamental definitions and properties of extriangulated categories that will be essential for our subsequent discussion. In Section 3, we introduce the concept of coresolution dimensions in extriangulated categories. For a recollement $(\A,\B,\C)$ of extriangulated categories, we establish connections between coresolution dimensions and certain subcategories, as presented in Theorem \ref{main-dim}. As an application, we extend our findings to module categories. Moving on to Section 4, we showcase how a (hereditary) cotorsion pair in $\B$ can trigger the emergence of (hereditary) cotorsion pairs in $\A$ and $\C$, as highlighted in Theorem \ref{main-cotor}. Additionally, we illustrate that a (hereditary) cotorsion triple in $\B$ can lead to the existence of (hereditary) cotorsion triples in $\A$ and $\C$, as outlined in Proposition \ref{prop-cot-trip}.
Finally, in Section 5, we present a method for constructing a recollement of abelian categories based on a recollement of extriangulated categories utilizing the hearts of cotorsion pairs, as detailed in Theorem \ref{main-construc}.

Throughout this paper, we assume that all subcategories are full, additive, and closed under isomorphisms.

\section{Preliminaries}
Let's take a moment to recap some key definitions and fundamental characteristics of extriangulated categories as presented in \cite{Na}. While we may skip certain particulars in this context, interested readers can delve into the comprehensive coverage available in \cite{Na}.

Consider an additive category $\mathcal{C}$ equipped with an additive bifunctor
$$\mathbb{E}: \mathcal{C}^{\mathrm{op}} \times \mathcal{C} \rightarrow \mathrm{Ab},$$ where $\mathrm{Ab}$ represents the category of abelian groups. For any objects $A, C \in \mathcal{C}$, an element $\delta \in \mathbb{E}(C,A)$ is referred to as an \emph{$\mathbb{E}$-extension}.

Let $\mathfrak{s}$ be a correspondence which associates an equivalence class
\[
\mathfrak{s}(\delta) = \xymatrix@C=0.8cm{[A \ar[r]^x & B \ar[r]^y & C]}
\]
to any $\mathbb{E}$-extension $\delta \in \mathbb{E}(C, A)$. This $\mathfrak{s}$ is termed a \emph{realization} of $\mathbb{E}$, if it ensures the commutativity of the diagrams defined in \cite[Definition 2.9]{Na}.
A triplet $(\mathcal{C}, \mathbb{E}, \mathfrak{s})$ is denoted as an \emph{extriangulated category} if it satisfies the following conditions:
\begin{itemize}
\item[(1)] $\mathbb{E} \colon \mathcal{C}^{\mathrm{op}} \times \mathcal{C} \rightarrow \mathrm{Ab}$ is an additive bifunctor.
\item[(2)] $\mathfrak{s}$ is an additive realization of $\mathbb{E}$.
\item[(3)] $\mathbb{E}$ and $\mathfrak{s}$ satisfy the compatibility conditions specified in \cite[Definition 2.12]{Na}.
\end{itemize}
For simplicity, we use the notation $\mathcal{C} =(\mathcal{C}, \mathbb{E}, \mathfrak{s})$.
\vspace{2mm}

We recall the following notations from \cite{Na}.

\begin{itemize}
\item[(1)] A sequence $\xymatrix@C=15pt{A\ar[r]^{x} & B \ar[r]^{y} & C}$ is called a {\it conflation} if it realizes some $\E$-extension $\del\in\E(C,A)$.
In this case, $\xymatrix@C=15pt{A\ar[r]^{x} & B}$ is called an {\it inflation} and $\xymatrix@C=15pt{B \ar[r]^{y} & C}$ is called a {\it deflation}. We call $$\xymatrix{A\ar[r]^{x} & B \ar[r]^{y} & C\ar@{-->}[r]^{\del}&}$$ an \emph{$\E$-triangle}.
We usually do not write this $``\delta"$ if it is not used in the argument.

\item[(2)]
Let $\X, \Y$ be subcategories of $\mathcal{C}$.  We denote by
\begin{align*}
{\rm Cone}(\X,\Y):=\{C\in \C\mid \text{ there is an }\E\text{-triangle}\xymatrix@C=15pt{X\ar[r] & Y \ar[r] & C\ar@{-->}[r]&}\text{ with }X\in \X\text{ and }Y\in \Y \};\\
{\rm Cocone}(\X,\Y):=\{C\in \C\mid \text{ there is an }\E\text{-triangle}\xymatrix@C=15pt{C\ar[r] & X \ar[r] & Y\ar@{-->}[r]&}\text{ with }X\in \X\text{ and }Y\in \Y \}.
\end{align*}
\end{itemize}

Throughout this paper,
for an extriangulated category $\C$, we assume the following condition, which is is equivalent to $\C$ being weakly idempotent complete (see \cite[Proposition 2.7]{K}).

\begin{condition}\label{WIC} {\rm \bf (WIC)} {\rm (see \cite[Condition 5.8]{Na})} Let $f:X\rightarrow Y$ and  $g:Y\rightarrow Z$ be any composable pair of morphisms in $\mathcal{C}$.
\begin{itemize}
\item[\rm (1)]  If $gf$ is an inflation, then $f$ is an inflation.
\item[\rm (2)] If $gf$ is a deflation, then $g$ is a deflation.
\end{itemize}
\end{condition}

\begin{definition}{\rm(\cite[Definitions 3.23 and 3.25]{Na})
Let $\mathcal{C}$ be an extriangulated category.
\begin{itemize}
\item[(1)] An object $P$ in $\mathcal{C}$ is called {\em projective} if for any $\mathbb{E}$-triangle $A\stackrel{x}{\longrightarrow}B\stackrel{y}{\longrightarrow}C\stackrel{}\dashrightarrow$ and any morphism $c$ in $\mathcal{C}(P,C)$, there exists $b$ in $\mathcal{C}(P,B)$ such that $yb=c$.
We denote the subcategory of projective objects in $\mathcal{C}$ by $\mathcal{P}(\mathcal{C})$.
Dually, the {\em injective} objects are defined, and the subcategory of injective objects in $\mathcal{C}$ is denoted by $\mathcal{I}(\mathcal{C})$.
\item[(2)] We say that $\mathcal{C}$ {\em has enough projectives} if for any object $M\in\mathcal{C}$, there exists an $\mathbb{E}$-triangle $$A\stackrel{}{\longrightarrow}P\stackrel{}{\longrightarrow}M\stackrel{}\dashrightarrow$$ satisfying $P\in\mathcal{P}(\mathcal{C})$. Dually, we define that $\mathcal{C}$ {\em has enough injectives}.
\end{itemize}}
\end{definition}

In the following, we will always assume that all extriangulated categories admit enough projective objects and enough injective objects.

In extriangulated categories,
the notions of the left (right) exact sequences (resp., functor) can be founded in \cite[Definitions 2.9 and 2.12]{WWZ20R} for details.
Also the notion of  extriangulated (resp., exact) functor between two extriangulated categories can be found in \cite{Ben} (resp., \cite[Definition 2.13]{WWZ20R}).

The notion of compatible morphisms is as follows.
\begin{definition}\label{right}{\rm(\cite[Definition 2.8]{WWZ20R})} Let $\mathcal{C}$ be an extriangulated category.
A morphism $f$ in $\mathcal{C}$ is called {\em compatible} provided that the following condition holds:
\begin{center}
 $f$ is both an inflation and a deflation implies that $f$ is an isomorphism.
\end{center}
That is, the class of compatible morphisms is the  class
$$\{f~|~f~\text{is~not~an~inflation},~\text{or}~f~\text{is~not~a~deflation},~\text{or}~f~\text{is~an~isomorphism}\}.$$
\end{definition}

\begin{remark}{\rm (\cite[Remark 2.15]{WWZ20R})}
Let $F\colon\A \to \B$ be an exact functor between extriangulated categories.
If $\A$ and $\B$ are abelian, then the right (resp., left) exact functor coincides with the
usual right exact functor in abelian categories, and the exact functor coincides with the
usual exact functor. If $\A$ and $\B$ are triangulated, then $F$ is left exact if and only if $F$ is a triangle
functor if and only if $F$ is right exact.
\end{remark}

Now we recall the concept of recollements of extriangulated categories \cite{WWZ20R}, which gives a simultaneous generalization of recollements of triangulated categories and abelian categories (see \cite{BBD, Fr}).

\begin{definition}\label{def-rec}{\rm(\cite[Definition 3.1]{WWZ20R}) }
Let $\mathcal{A}$, $\mathcal{B}$ and $\mathcal{C}$ be three extriangulated categories. A \emph{recollement} of $\mathcal{B}$ relative to
$\mathcal{A}$ and $\mathcal{C}$, denoted by ($\mathcal{A}$, $\mathcal{B}$, $\mathcal{C}$), is a diagram
\begin{equation}\label{recolle}
  \xymatrix{\mathcal{A}\ar[rr]|{i_{*}}&&\ar@/_1pc/[ll]|{i^{*}}\ar@/^1pc/[ll]|{i^{!}}\mathcal{B}
\ar[rr]|{j^{\ast}}&&\ar@/_1pc/[ll]|{j_{!}}\ar@/^1pc/[ll]|{j_{\ast}}\mathcal{C}}
\end{equation}
given by two exact functors $i_{*},j^{\ast}$, two right exact functors $i^{\ast}$, $j_!$ and two left exact functors $i^{!}$, $j_\ast$, which satisfies the following conditions:
\begin{itemize}
  \item [\rm (R1)] $(i^{*}, i_{\ast}, i^{!})$ and $(j_!, j^\ast, j_\ast)$ are adjoint triples.
  \item [\rm (R2)] $\Im i_{\ast}=\Ker j^{\ast}$.
  \item [\rm (R3)] $i_\ast$, $j_!$ and $j_\ast$ are fully faithful.
  \item [\rm (R4)] For each $B\in\mathcal{B}$, there exists a left exact $\mathbb{E}$-triangle sequence
$$
  \xymatrix{i_\ast i^!( B)\ar[r]^-{\theta_B}&B\ar[r]^-{\vartheta_B}&j_\ast j^\ast (B)\ar[r]&i_\ast (A)}
$$
 in $\mathcal{B}$ with $A\in \mathcal{A}$, where $\theta_B$ and  $\vartheta_B$ are given by the adjunction morphisms.
  \item [\rm (R5)] For each $B\in\mathcal{B}$, there exists a right exact $\mathbb{E}$-triangle sequence
$$
  \xymatrix{i_\ast\ar[r]( A') &j_! j^\ast (B)\ar[r]^-{\upsilon_B}&B\ar[r]^-{\nu_B}&i_\ast i^\ast (B)&}
$$
in $\mathcal{B}$ with $A'\in \mathcal{A}$, where $\upsilon_B$ and $\nu_B$ are given by the adjunction morphisms.
\end{itemize}
\end{definition}
In particular,
if $\A$, $\B$ and $\C$ are abelian categories or triangulated categories, then $(\A,\B,\C)$ is a recollement if it satisftis the following conditions.
\begin{itemize}
\item[(a)] $(i^*, i_*, i^!)$ and $(j_!, j^*, j_*)$ are adjoint triples.
\item[(b)] the functors $i_{*}, j_{!}$ and $j_{*}$ are fully faithful.
\item[(c)] $j^{*}i_{*}=0$.
\end{itemize}
\vspace{2mm}

We collect some properties of recollements of extriangulated categories (see \cite{WWZ20R}).

\begin{lemma}\label{lem-rec} Let ($\mathcal{A}$, $\mathcal{B}$, $\mathcal{C}$) be a recollement of extriangulated categories as in {\rm (\ref{recolle}).}

$(1)$ All the natural transformations
$$i^{\ast}i_{\ast}\Rightarrow\Id_{\A},~\Id_{\A}\Rightarrow i^{!}i_{\ast},~\Id_{\C}\Rightarrow j^{\ast}j_{!},~j^{\ast}j_{\ast}\Rightarrow\Id_{\C}$$
are natural isomorphisms.
Moreover, $i^{!}$, $i^{*}$ and $j^{*}$ are dense.

$(2)$ $i^{\ast}j_!=0$ and $i^{!}j_\ast=0$.

$(3)$ $i^{\ast}$ preserves projective objects and $i^{!}$ preserves injective objects.

$(3')$ $j_{!}$ preserves projective objects and $j_{\ast}$ preserves injective objects.

$(4)$ If $i^{!}$ (resp., $j_{\ast}$) is  exact, then $i_{\ast}$ (resp., $j^{\ast}$) preserves projective objects.

$(4')$ If $i^{\ast}$ (resp., $j_{!}$) is  exact, then $i_{\ast}$ (resp., $j^{\ast}$) preserves injective objects.

$(5)$ If $i^{!}$ is exact, then $j_{\ast}$ is exact.

$(5')$
If $i^{\ast}$ is exact, then $j_{!}$ is  exact.

$(6)$ If $i^{!}$ is exact, for each $B\in\mathcal{B}$, there is an $\mathbb{E}$-triangle
  \begin{equation*}\label{third}
  \xymatrix{i_\ast i^! (B)\ar[r]^-{\theta_B}&B\ar[r]^-{\vartheta_B}&j_\ast j^\ast (B)\ar@{-->}[r]&}
   \end{equation*}
 in $\mathcal{B}$ where $\theta_B$ and  $\vartheta_B$ are given by the adjunction morphisms.

$(6')$ If $i^{\ast}$ is exact, for each $B\in\mathcal{B}$, there is an $\mathbb{E}$-triangle
  \begin{equation*}\label{four}
  \xymatrix{ j_! j^\ast (B)\ar[r]^-{\upsilon_B}&B\ar[r]^-{\nu_B}&i_\ast i^\ast (B) \ar@{-->}[r]&}
   \end{equation*}
in $\mathcal{B}$ where $\upsilon_B$ and $\nu_B$ are given by the adjunction morphisms.
\end{lemma}

For a subcategory $\X\subseteq\C$, put $\Omega^0\X=\X$, and define $\Omega^k\X$ for $k>0$ inductively by
$$\Omega^k\X=\Omega(\Omega^{k-1}\X)={\rm Cocone}(\P,\Omega^{k-1}\X).$$
We call $\Omega^k\X$ the $k$-th syzygy of $\X$.

Dually we define the $k$-th cosyzygy $\Sigma^k\X$ by
$\Sigma^0\X=\X$ and $\Sigma^k\X={\rm Cone}(\Sigma^{k-1}\X,\I)$ for $k>0$.

Liu and Nakaoka \cite{Liu} defined higher extension groups in an extriangulated category as $$\E(A,\Sigma^kB)\simeq\E(\Omega^kA,B)~~\mbox{for}~~k\geq0.$$
For convenience, we denote $\E(A,\Sigma^kB)\simeq\E(\Omega^kA,B)$ by $\E^{k+1}(A,B)$ for $k\geq0$.
They proved the following result, which is an important tool in relative homological theory of extriangulated categories and can be called ``dimension shifting" as a strategy.
\begin{lemma}
Let $\C$ be an extriangulated category. Assume that
$$\xymatrix{A\ar[r]^{f}& B\ar[r]^{g}&C\ar@{-->}[r]^{\delta}&}$$
is an $\E$-triangle in $\C$. Then for any object $X\in\C$ and $k\geq1$, we have the following exact sequences:
$$\cdots\xrightarrow{}\E^k(X,A)\xrightarrow{}\E^k(X,B)\xrightarrow{}\E^k(X,C)\xrightarrow{}\E^{k+1}(X,A)\xrightarrow{}\E^{k+1}(X,B)\xrightarrow{}\cdots;$$
$$\cdots\xrightarrow{}\E^k(C,X)\xrightarrow{}\E^k(B,X)\xrightarrow{}\E^k(A,X)\xrightarrow{}\E^{k+1}(C,X)\xrightarrow{}\E^{k+1}(B,X)\xrightarrow{}\cdots.$$
\end{lemma}

\section{Coresolution dimensions relative to a subcategory}

Recall from \cite{ZZ20T} that an {\em $\mathbb{E}$-triangle sequence} is defined as a sequence
 $$\cdots {\longrightarrow}X_{n+2}\stackrel{d_{n+2}}{\longrightarrow}X_{n+1}\stackrel{d_{n+1}}{\longrightarrow} X_{n}\stackrel{d_{n}}{\longrightarrow}X_{n-1}{\longrightarrow}\cdots $$
 in $\mathcal{C}$ such that for any $n$,
 there are $\mathbb{E}$-triangles $K_{n+1}\stackrel{g_{n}}{\longrightarrow}X_{n}\stackrel{f_{n}}{\longrightarrow}K_{n}\stackrel{}\dashrightarrow$ and the differential $d_n=g_{n-1}f_n$.

 As a dual of \cite[Definition 4.1]{MZZ}, we introduce the notion of coresolution dimension for a subcategory of $\mathcal{C}$.
 \begin{definition}
{\rm Let $\X$ be a subcategory of $\mathcal{C}$ and $C$ an object in $\mathcal{C}$.
The {\bf $\X$-coresolution dimension} of $C$, written $\X$-$\cores C$, is defined by
\begin{align*}
\X\text{-}\cores C=&\inf \{n \geq 0\mid\text{ there exists an } \mathbb{E}\text{-triangle sequence}\\
&\xymatrix@C=15pt
{0\ar[r]&C\ar[r]&X_{0}\ar[r]&\cdots\ar[r]&X_{n-1}\ar[r]&X_{n}\ar[r]&0} \text{ in } \mathcal{C} \text{ with } X_{i}\in \X \text{ for }0\leq i\leq n\}.
\end{align*}
For a subcategory $\C'$ of $\C$, the coresolution dimension of $\mathcal{C'}$, denoted by $\X\text{-}\cores \mathcal{C'}$, is defined as
$$\X\text{-}\cores  \mathcal{C'}:=\sup\{\X\text{-}\cores C\mid C\in \mathcal{C'}\};$$}
in particular, $\X\text{-}\cores  \mathcal{C}$ is defined.
\end{definition}
Let $\X$ be a subcategory of $\C$, we define two subcategories as follows   %%associated
$$\X^\perp:=\{C\in \C\mid \E^{i}(X,C)=0 \text{ for any }i\geq 1 \text{ and }X\in \X\}.$$
$$^\perp\X:=\{C\in \C\mid \E^{i}(C,X)=0 \text{ for any }i\geq 1 \text{ and }X\in \X\}.$$
Clearly, $\I(\C)\subseteq\X^\perp$. In particular, if $\X=\C$, then $\X^\perp=\I(\C)$; if $\X=\P(\C)$, then $\X^\perp=\C$.

Here we obtain the following characterization of the coresolution dimensions related to the subcategory $\X^\perp$.

\begin{lemma}\label{lem-eq}
Let $\X$ be a subcategory of $\C$. For any $C\in \C$ and an integer $n\geq 0$, the following statements are equivalent.
\begin{itemize}
\item[\rm (1)] $\X^\perp\text{-}\cores C\leq n$.
\item[\rm (2)] $\E^{n+k}(X,C)=0$ for any $X\in \X$ and any $k\geq 1$.
\item[\rm (3)] if $\xymatrix@C=15pt
{0\ar[r]&C\ar[r]&Y_{0}\ar[r]&\cdots\ar[r]&Y_{n-1}\ar[r]&Y_{n}\ar[r]&0}$ is an $\E$-triangle sequence with $Y_{i}\in \X^\perp$ for any $0\leq i\leq n-1$, then $Y_{n}\in\X^\perp$.
\end{itemize}
\end{lemma}
\begin{proof}
$(1)\Rightarrow(2)$
Suppose $\X^\perp\text{-}\cores C= n$, there exists an $\E$-triangle
$$\xymatrix@C=15pt
{0\ar[r]&C\ar[r]&Y_{0}\ar[r]&\cdots\ar[r]&Y_{n-1}\ar[r]&Y_{n}\ar[r]&0}$$
with all $Y_{i}\in \X^\perp$.
For any $X\in\X$, one can obtain the assertion that $\E^{n+k}(X,C)\cong\E^{k}(X,Y_{n})=0$ for any $k\geq 1$.

$(2)\Rightarrow(3)$ It is easy to check that $\E^{k}(X,Y_{n})\cong\E^{n+k}(X,C)=0$ for any $X\in \X$ and $k\geq 1$, so $Y_{n}\in\X^\perp$.

$(3)\Rightarrow(1)$ It is obvious.
\end{proof}
If $\C=\Mod A$ is the module category consisting of all right $A$-modules for a ring $A$, $\X$ denotes the subcategory consisting of all flat $A$-modules, then any $A$-module in $\X^\perp$ is called cotorsion in sense of \cite{MD06}, in this case, for any $A$-module $C$, $\X^\perp\text{-}\cores C$ coincides to the cotorsion dimension $\cotd C$ by Lemma \ref{lem-eq}, and $\X^\perp\text{-}\cores \C$ coincides to $\cotd \Mod A$ (in \cite{MD06}, it is called the right global cotorsion dimension).
Following \cite[Corollary 19.2.9]{MD06}, we know that $\cotd \Mod A=0$ if and only if $A$ is right perfect.

The following result is useful in the sequel.
\begin{lemma}\label{lem-3term}
Let $\X$ be a subcategory of $\C$ and
$$\xymatrix{A\ar[r]&B\ar[r]&C\ar[r]&D}$$
be an $\mathbb{E}$-triangle sequence in $\mathcal{C}$. Then
\begin{itemize}
\item[\rm (1)] If $D=0$, then
\begin{itemize}
\item[\rm (a)] $\X^\perp\text{-}\cores B\leq\max \{\X^\perp\text{-}\cores A,\X^\perp\text{-}\cores C\}$.
\item[\rm (b)] $\X^\perp\text{-}\cores C\leq\max\{\X^\perp\text{-}\cores A-1, \X^\perp\text{-}\cores B\}$.
\item[\rm (c)] $\X^\perp\text{-}\cores A\leq\max\{\X^\perp\text{-}\cores B, \X^\perp\text{-}\cores C+1\}$.
\end{itemize}
\item[\rm (2)] If $D\neq 0$, then
$$\X^\perp\text{-}\cores B\leq\max\{\X^\perp\text{-}\cores A, \X^\perp\text{-}\cores C, \X^\perp\text{-}\cores D+1\}.$$
\end{itemize}
\end{lemma}
\begin{proof}
(1) For any $X\in \X$ and any $i\geq 1$, we have the following exact sequence
\begin{align*}
\xymatrix@C=15pt{\cdots \ar[r]&\E^{i}(X,A)\ar[r]&\E^{i}(X,B)\ar[r]&\E^{i}(X,C)\ar[r]&\E^{i+1}(X,A)\ar[r]&\E^{i+1}(X,B)\ar[r]&\E^{i+1}(X,C)\ar[r]&\cdots}
\end{align*}
By a similar argument in classical homological algebra, the assertions follow from Lemma \ref{lem-eq}.

(2) It follows from (1).
\end{proof}

\begin{lemma}\label{lem-j_{!}}
Let $(\mathcal{A},\mathcal{B},\mathcal{C})$ be a recollement of extriangulated categories and $\mathcal{X}$ and $\mathcal{X''}$ be subcategories of $\mathcal{B}$ and
$\mathcal{C}$ respectively.
Then we have
$$\X^\perp\text{-} \cores j_{*}(C) \leq  \X''^\perp\text{-}\cores C+ \max\{\X^\perp\text{-}\cores i_{*}(\A), \X^\perp\text{-}\cores j_{*}(\X''^\perp) \}+ 1.$$
\end{lemma}
\begin{proof}
If $\X''^\perp\text{-}\cores C=\infty$ or $\max\{\X^\perp\text{-}\cores i_{*}(\mathcal{A}), \X^\perp\text{-}\cores j_{*}(\X''^\perp)\}=\infty$, there is nothing to prove.
Assume that
$\max\{\X^\perp\text{-}\cores i_{*}(\mathcal{A}), \X^\perp\text{-}\cores j_{*}(\X''^\perp)\}=m$.
The proof will be proceed by induction on the $\X''^\perp\text{-}\cores C$.
If $C\in \X''^\perp$, the assertion holds obviously.
Now suppose that $\X''^\perp\text{-}\cores C=n\geq 1$.
By Lemma \ref{lem-eq}, we have the following $\mathbb{E}$-triangle sequence

$$\xymatrix@C=15pt{0\ar[rr]&&C\ar[rr]^{f}&&I_{0}\ar[rd]^{g}\ar[rr]&&I_{1}\ar[rr]&&\cdots\ar[rr]&&I_{n-1}\ar[rr]&&
Y_{n}\ar[rr]&&0\\
&&&&&K_{1}\ar[ru]&&&}$$
with all $I_{i}\in \mathcal{I(C)}\subseteq \X''^\perp$ and $Y_{n}\in \X''^\perp$.

Notice that $\X''^\perp\text{-}\cores K_{1}\leq n-1$, by induction hypothesis, we have
$$\X''^\perp\text{-}\cores j_{*}(K_{1})\leq \X''^\perp\text{-}\cores K_{1}+m+1\leq n-1+m+1=m+n.$$
Since $j_{*}$ is left exact, there is an $\mathbb{E}$-triangle $\xymatrix{j_{*}(C)\ar[r]^{j_*(f)}&j_{*}(I_{0})\ar[r]^{h_2}&K_{1}'\ar@{-->}[r]&}$
in $\mathcal{B}$ and an inflation $\xymatrix{h_{1}: K'_{1}\ar[r]& j_{*}(K_{1})}$ which is compatible, such that $j_{*}(g)=h_{1}h_{2}$.

Since $j^{*}j_{*}\cong {\rm Id}_{\mathcal{C}}$, by Lemma \ref{lem-rec},
$$\xymatrix{j^*j_{*}(C)\ar[r]^{j^*j_*(f)}&j^*j_{*}(I_{0})\ar[r]^{j^*(h_2)}&j^{*}(K_{1}')\ar@{-->}[r]&}$$
is an $\mathbb{E}$-triangle in $\mathcal{C}$.
Since $g=j^{*}j_{*}(g)=(j^{*}(h_{1}))(j^{*}(h_{2}))$, so $j^{*}(h_{1})$ is a deflation by Condition \ref{WIC}.
Note that $j^{*}(h_{1})$ is an inflation and compatible since $j^{*}$ is exact, we have that $j^{*}(h_{1})$ is an isomorphism.
Thus  $ j^{*}(K'_{1})\cong j^{*}j_{*}(K_{1})$.
Set $K''_{1}=\cone (h_{1})$, consider the following $\mathbb{E}$-triangle
\begin{align}\label{E-triangle-1}
\xymatrix@C=20pt{K'_{1}\ar[r]^{h_1}&j_{*}(K_{1})\ar[r]&K''_{1}\ar@{-->}[r]&}
\end{align}
in $\mathcal{B}$.
Since $j^{*}$ is exact,  $j^{*}(K''_{1})=0$. By (R2), there exists an object $A'\in \mathcal{A}$ such that $K''_{1}\cong i_{*}(A')$.
Then $\X^\perp\text{-}\cores K''_{1}\leq m$ by assumption.
Apply Lemma \ref{lem-3term} to
% the $\mathbb{E}$-triangle
(\ref{E-triangle-1}),
we have $\X^\perp\text{-}\cores K'_{1}\leq m+n$.
Notice that $j_{*}$ preserves injectives by Lemma \ref{lem-rec}, so $j_{*}(I_{0})\in \mathcal{I(B)}(\subseteq \X^\perp)$.
It follows that $\X^\perp\text{-}\cores  j_{*}(C)\leq m+n+1$.
\end{proof}

Now we state and prove our main theorem in this section.

\begin{theorem}\label{main-dim}
Let $(\mathcal{A},\mathcal{B},\mathcal{C})$ be a recollement of extriangulated categories and $\X'$, $\X$ and $\X''$ be subcategories
of $\mathcal{A},\ \mathcal{B}$ and $\mathcal{C}$, respectively. Then we have the following statements hold.
\begin{itemize}
\item[\rm (1)] $\X^\perp\text{-}\cores \B \leq  \X''^\perp\text{-}\cores \C+\max\{\X^\perp\text{-}\cores  i_{*}(\A), \X^\perp\text{-}\cores  j_{*}(\X''^\perp)\} + 1.$
\item[\rm (2)] $\X^\perp\text{-}\cores i_{*}(\A)\leq \X'^\perp\text{-}\cores \A+\X^\perp\text{-}\cores i_{*}(\X'^\perp)$.
\item[\rm (3)]
 $\X^\perp\text{-}\cores \B \leq$\\
 \hfill $\X'^\perp\text{-}\cores \A+ \X''^\perp\text{-}\cores \mathcal{C}+\max\{\X^\perp\text{-}\cores i_{*}(\X'^\perp), \X^\perp\text{-}\cores  j_{*}(\X''^\perp)\} + 1.$
\item[\rm (4)] $\X''^\perp\text{-}\cores \C\leq \X^\perp\text{-}\cores \B+ \X''^\perp\text{-}\cores  j^{*}(\X^\perp)$.
\item[\rm (5)] if $j_{*}(\X''^\perp)\subseteq \X^\perp$ and $i_{*}(\X'^\perp)\subseteq\X^\perp$, then
$$\X^\perp\text{-}\cores \B \leq \X'^\perp\text{-}\cores \A+ \X''^\perp\text{-}\cores \mathcal{C} + 1.$$
\item[\rm (6)] if $j^{*}(\X^\perp)\subseteq \X''^\perp$, then $\X''^\perp\text{-}\cores \C\leq \X^\perp\text{-}\cores \B$.

   \item[\rm (7)] if $i^!$ is exact and $i^{!}(\X^\perp)\subseteq\X'^\perp$, then $\X'^\perp\text{-}\cores \A\leq \X^\perp\text{-}\cores \B$.

\item[\rm (8)] if $i^!$ is exact, and if $j_{*}(\X''^\perp)\subseteq \X^\perp$, $j^{*}(\X^\perp)\subseteq \X''^\perp$, $i_{*}(\X'^\perp)\subseteq\X^\perp$ and $i^{!}(\X^\perp)\subseteq\X'^\perp$, then
$$\hspace{-6mm}\max\{\X'^\perp\text{-}\cores \A,\X''^\perp\text{-}\cores \mathcal{C}\}\leq \X^\perp\text{-}\cores \B \leq \X'^\perp\text{-}\cores \A+ \X''^\perp\text{-}\cores \mathcal{C} + 1.$$
\end{itemize}
\end{theorem}

\begin{proof}

(1)
Suppose that
$$\max\{\X^\perp\text{-}\cores i_{*}(\A),\X^\perp\text{-}\cores j_{*}(\X''^\perp)\}=m < \infty~~\mbox{and}~~\X''^\perp\text{-}\cores \C= n <\infty.$$
Let $B\in \B$.
By (R4),
there exists a commutative diagram
\begin{equation*}
\xymatrix{
  &i_{*}i^{!}(B)\ar[rr]^-{\theta_B}&  &B\ar[dr]_{h_{2}}\ar[rr]^-{\nu_B}&&j_{*}j^{*}(B) \ar[r]&i_{*}(A') \\
           &            &&    &         B' \ar[ur]_{h_{1}}& }
\end{equation*}
in $\mathcal{B}$ such that $i_{*}i^{!}(B){\longrightarrow}B\stackrel{h_{2}}{\longrightarrow}B'\stackrel{}\dashrightarrow$ and $B'\stackrel{h_{1}}{\longrightarrow}j_{*}j^{*}(B) {\longrightarrow}i_{*}(A')\stackrel{}\dashrightarrow$ are $\mathbb{E}$-triangles and $h_{1}$ is compatible.
Notice that
$\X^\perp\text{-}\cores i_{*}(A')\leq m$ and $\X^\perp\text{-}\cores i_{*}i^{!}(B)\leq m$.
By Lemmas \ref{lem-3term} and \ref{lem-j_{!}},
\begin{align*}
\X^\perp\text{-}\cores B&\leq \max\{\X^\perp\text{-}\cores i_{*}i^{!}(B), \X^\perp\text{-}\cores j_{*}j^{*}(B),\X^\perp\text{-}\cores i_{*}(A')+1\}\\
&\leq \max\{\X''^\perp\text{-}\cores j^{*}(B)+m+1,m+1,m\}
\end{align*}
Note that $ \X''^{\perp}\text{-}\cores j^{*}(B)\leq n$, so $\X^\perp\text{-}\cores B\leq m+n+1$.

(2)
 Suppose that $\X^\perp\text{-}\cores i_{*}(\X'^\perp)=n<\infty$ and $\X'^\perp\text{-}\cores \A=m< \infty$.
Let $A\in \mathcal{A}$.
If $A\in \X'^\perp$, then $\X^\perp\text{-}\cores  i_{*}(A)\leq n$ and our result holds.
Now suppose that $\X'^\perp\text{-}\cores A=s\leq m$.
Consider the following $\mathbb{E}$-triangle sequence
$$\xymatrix{A\ar[r]&X'_{0}\ar[r]&\cdots\ar[r]&X'_{s-1}\ar[r]&X'_{s}}$$
in $\mathcal{A}$ with $X'_{i}\in \X'^\perp$ for $0\leq i\leq s$.
Since $i_{*}$ is exact,
$$\xymatrix{i_{*}(A)\ar[r]&i_{*}(X'_{0})\ar[r]&\cdots\ar[r]&i_{*}(X'_{s-1})\ar[r]&i_{*}(X'_{s})}$$
is an $\mathbb{E}$-triangle sequence in $\mathcal{B}$.
By assumption, we know
$\X^\perp\text{-}\cores  i_{*}(X'_{i})\leq n$.
Then
$$\X^\perp\text{-}\cores  i_{*}(A)\leq s+n \leq m+n$$
by Lemma \ref{lem-3term}.

(3) It follows from (1) and (2).

(4)
Suppose that $\X^\perp\text{-}\cores \mathcal{B}= m<\infty$ and $\X''^\perp\text{-}\cores j^{*}(\X^\perp)=n<\infty$.
Let $C\in \mathcal{C}$. We know $j_{!}(C)\in \mathcal{B}$.
Assume that $\X^\perp\text{-}\cores j_!(C)=s\leq m$,
and consider the following $\mathbb{E}$-triangle sequence
$$\xymatrix{j_{!}(C)\ar[r]&X_{0}\ar[r]&X_{1}\ar[r]&\cdots\ar[r]&X_{s-1}\ar[r]&X_{s}}$$
in $\B$ with $X_{i}\in \X^\perp$ for $0\leq i\leq s$.

Since $j^{*}$ is exact and $j^*j_!\cong \Id_\C$,
$$\xymatrix{C\ar[r]&j^{*}(X_{0})\ar[r]&j^{*}(X_{1})\ar[r]&\cdots\ar[r]&j^{*}(X_{s-1})\ar[r]&j^{*}(X_{s})}$$
is an $\mathbb{E}$-triangle sequence in $\mathcal{C}$.
Notice that $\X''^\perp\text{-}\res j^{*}(X_{i})\leq n$ by assumption,
so $\X''^\perp\text{-}\res C\leq s+n \leq m+n$ by Lemma \ref{lem-3term}.

(5) It follows from (3).

(6) It follows from (4).

(7)
For any object $A\in\A$, $i_{*}(A)\in \B$.
Suppose $\X^\perp\text{-}\cores B=n$, so there exists an $\E$--triangle sequence
$$\xymatrix{i_{*}(A)\ar[r]&X_{0}\ar[r]&X_{1}\ar[r]&\cdots\ar[r]&X_{s-1}\ar[r]&X_{s}}$$
in $\B$ with $X_{i}\in \X^\perp$ for $0\leq i\leq n$.
The exactness of $i^{!}$ yields the following $\E$--triangle sequence
$$\xymatrix{A(\cong i^!i_{*}(A)) \ar[r]&i^!(X_{0})\ar[r]&i^!(X_{1})\ar[r]&\cdots\ar[r]&i^!(X_{s-1})\ar[r]&i^!(X_{s})}.$$
Notice that $i^{!}(\X^\perp)\subseteq\X'^\perp$ by assumption, so $\X'^\perp\text{-}\cores A\leq n$, and thus $\X'^\perp\text{-}\cores A \leq \X^\perp\text{-}\cores \B$.

(8) It follows from (5)(6)(7).
\end{proof}

Taking $\X'=\A$, $\X=\B$ and $\X''=\C$, applying Theorem \ref{main-dim} to abelian categories yields a dual of \cite[Theorem 4.1 and Proposition 4.4]{PC14H}.
\begin{corollary}
Let $(\mathcal{A},\mathcal{B},\mathcal{C})$ be a recollement of abelian categories. Then the following statements hold.
\begin{itemize}
\item[\rm (1)] $\I(\B)\text{-}\cores \B \leq  \I(\C)\text{-}\cores \C+\I(\B)\text{-}\cores  i_{*}(\A) + 1.$
\item[\rm (2)] $\I(\B)\text{-}\cores i_{*}(\A)\leq \I(\A)\text{-}\cores \A+\I(\B)\text{-}\cores i_{*}(\I(\A))$.
\item[\rm (3)] $\I(\B)\text{-}\cores \B \leq \I(\A)\text{-}\cores \A+ \I(\C)\text{-}\cores \mathcal{C}+\I(\B)\text{-}\cores i_{*}(\I(\A)) + 1.$
\item[\rm (4)] $\I(\C)\text{-}\cores \C\leq \I(\B)\text{-}\cores \B+ \I(\C)\text{-}\cores  j^{*}(\I(\B))$.
\item[\rm (5)] if $i_{*}(\I(\A))\subseteq\I(\B)$, then
$\I(\B)\text{-}\cores \B \leq \I(\A)\text{-}\cores \A+ \I(\C)\text{-}\cores \mathcal{C} + 1.$
\item[\rm (6)] if $j^{*}(\I(\B))\subseteq \I(\C)$, then $\I(\C)\text{-}\cores \C\leq \I(\B)\text{-}\cores \B$.
 \item[\rm (7)] if $i^!$ is exact, then $\I(\A)\text{-}\cores \A\leq \I(\B)\text{-}\cores \B$.
 \item[\rm (8)] if $i^!$ is exact, and if $j^{*}(\I(\B))\subseteq \I(\C)$, $i_{*}(\I(\A))\subseteq\I(\B)$, then
$$\hspace{-8mm}\max\{\I(\A)\text{-}\cores \A,\I(\C)\text{-}\cores \mathcal{C}\}\leq \I(\B)\text{-}\cores \B \leq \I(\A)\text{-}\cores \A+ \I(\C)\text{-}\cores \mathcal{C} + 1.$$
\end{itemize}
\end{corollary}
\begin{proof}
The assertions follow from Theorem \ref{main-dim} and the fact that the functors $i^{!}$ and $j_*$ preserve injective objects by Lemma \ref{lem-rec}.
\end{proof}

By applying Theorem \ref{main-dim} to module categories, we have the following result.

\begin{corollary}
{\rm (cf. \cite[Theorem 2.7 and Corollary 3.5]{FH22T})}
Let $(\Mod A,\Mod B,\Mod C)$ be a recollement of module categories for rings $A$, $B$ and $C$,
$$
  \xymatrix{\Mod A\ar[rr]|{i_{*}}&&\ar@/_1pc/[ll]|{i^{*}}\ar@/^1pc/[ll]|{i^{!}}\Mod B
\ar[rr]|{j^{\ast}}&&\ar@/_1pc/[ll]|{j_{!}}\ar@/^1pc/[ll]|{j_{\ast}}\Mod C}
$$
and let $\X'$, $\X$ and $\X''$ be subcategories consisting of all flat modules
in $\Mod A,\Mod B$ and $\Mod C$ respectively.
Then
\begin{itemize}
\item[\rm (1)] $\cotd \Mod B \leq \cotd \Mod A+ \cotd \Mod C+\max\{\cotd i_{*}(\X'^\perp), \cotd  j_{*}(\X''^\perp)\} + 1.$
\item[\rm (2)] $\cotd \Mod C\leq \cotd \Mod B+ \cotd  j^{*}(\X^\perp)$.
\item[\rm (3)] if $j^{*}(\X^\perp)\subseteq \X''^\perp$, then $\cotd \Mod C\leq \cotd \Mod B$; in particular, if $B$ is right perfect, then so is $C$.
\item[\rm (4)] if $i^!$ is exact and $i^{!}(\X^\perp)\subseteq \X'^\perp$, then $\cotd \Mod A\leq \cotd \Mod B$; in particular, if $B$ is right perfect, then so is $C$.
\end{itemize}
\end{corollary}

\section{Cotorsion pairs and cotorsion triples}

Now we recall the notion of cotorsion pairs.

\begin{definition}\label{DefCotors}
Let $\X,\Y\subseteq\C$ be a pair of subcategories.
\begin{itemize}
\item[\rm (1)] {\rm (\cite[Definition 4.1]{Na})} The pair $(\X,\Y)$ is called a {\it cotorsion pair} in $\C$ if it satisfies the following conditions.
\begin{enumerate}
\item[\bf(CP1)] $\E(\X,\Y)=0$.
\item[\bf(CP2)] For any $C\in\C$, there exist two $\E$-triangles
$$\xymatrix@C=15pt{Y^1\ar[r]& X^1\ar[r]& C\ar@{-->}[r]&}$$
$$\xymatrix@C=15pt{C\ar[r]& Y_1\ar[r]& X_1\ar@{-->}[r]&}$$
satisfying $X^1, X_1 \in\X,Y^1, Y_1 \in\Y$.
\end{enumerate}
\item[\rm (2)] {\rm (\cite[Definition 3.1]{AT22H})} A cotorsion pair $(\mathcal{X}, \mathcal{Y})$ is called hereditary if it satisfies the following condition.
    \begin{itemize}
    \item[\bf(HCP)] $\mathbb{E}^i(\mathcal{X}, \mathcal{Y})=0$ for each $i \geq 2$.
    \end{itemize}
\end{itemize}
\end{definition}
If $(\X,\X)$ is a cotorsion pair in $\C$, the subcategory is called the cluster tilting subcategory of $\C$ in sense of (\cite[Remark 2.11]{ZZ19C}).
Three subcategories $\mathcal{X}, \mathcal{Y}$ and $\mathcal{Z}$ of $\C$ form a cotorsion triple $(\mathcal{X}, \mathcal{Y}, \mathcal{Z})$ if $(\mathcal{X}, \mathcal{Y})$ and $(\mathcal{Y}, \mathcal{Z})$ are cotorsion pairs in $\C$.
Moreover, a cotorsion triple $(\mathcal{X}, \mathcal{Y}, \mathcal{Z})$ in $\C$ is hereditary if $(\mathcal{X}, \mathcal{Y})$ and $(\mathcal{Y}, \mathcal{Z})$ are hereditary cotorsion pairs.

\begin{remark}
{\rm (see \cite[Remark 4.4]{Na})}\label{RemCotors}
Let $(\X,\Y)$ be a cotorsion pair in $\C$, the following statements hold for any $C\in\C$.
\begin{enumerate}
\item[\rm (1)] $C\in\X\Leftrightarrow\mathbb{E}(C,\Y)=0$.
\item[\rm (2)] $C\in\Y\Leftrightarrow\mathbb{E}(\X,C)=0$.
\end{enumerate}
 In addition, if $(\X,\Y)$ is hereditary, then $\Y=\X^{\perp}$ and $\X={^\perp\Y}$.
\end{remark}

The following results are easy and useful in the sequel.

\begin{lemma}\label{E-xt}
Let $\A$ and $\B$ be extriangulated categories, and let $F : \A \to \B$ be a functor admitting a right adjoint functor $G$.
For any $X\in \A$, $Y\in \B$ and any $i\geq 1$, if one of the following conditions is satisfied
 \begin{itemize}
 \item[\rm (1)] If $F$ is an exact functor and preserves projective objects;
 \item[\rm (2)] If $G$ is an exact functor and preserves injective objects;
 \end{itemize}
then we have
$$\mathbb{E}_{\B}^{i}(F(X),Y)\cong \mathbb{E}_{\A}^{i}(X,G(Y)).$$
\end{lemma}
\begin{proof}
For any $X\in\X$,
consider the following $\E$-triangle
$$\xymatrix@C=15pt{\Omega X\ar[r]&P\ar[r]&X{\dashrightarrow}}$$
in $\A$ with $P\in\P(\A)$.
If $i=1$, the assertion follows from \cite[Lemma 2.16]{WWZ20R}.
For $i>1$, and for any $Y\in \Y$, it is easy to see that $\E_{\B}^i(F(X), Y) \cong \E_{\B}^{i-1}(F(\Omega X),Y)$ and $\E_{\A}^i(X, G(Y)) \cong \E_{\A}^{i-1}(\Omega X, G(Y))$.
By induction hypothesis, we have
$$\E_{\B}^i(F(X), Y) \cong \E_{\A}^i(X, G(Y)).$$
\end{proof}

\begin{lemma}\label{lem-=}
Let $(\A, \B, \C)$ be a recollement of extriangulated categories, and
let $\X$ be a subcategory of $\mathcal{B}$. If $i_{*}i^{*}(\X) \subseteq \X$ and $i_{*}i^{!}(\X) \subseteq \X$, then $i^{*}(\X) =i^{!}(\X)$.
\end{lemma}
\begin{proof}
Since $i^{*}i_{*}\cong \Id_{\mathcal{A}}$ and $i_{*}i^{!}(\mathcal{X})\subseteq \mathcal{X}$, we have
$i^{!}(\mathcal{X})\subseteq i^{*}(\mathcal{X})$.
On the other hand, $i^{!}i_{*}\cong \Id_{\mathcal{A}}$ and $i_{*}i^{*}(\mathcal{X})\subseteq \mathcal{X}$ imply that
$i^{*}(\mathcal{X})\subseteq i^{!}(\mathcal{X})$. Then $i^{*}(\mathcal{X})= i^{!}(\mathcal{X})$.
\end{proof}

\begin{lemma}\label{lem-=if}
Let $(\A, \B, \C)$ be a recollement of extriangulated categories, and
let $(\X,\Y)$ be a cotorsion pair in $\B$. If $i^{*}$ and $i^!$ are exact, then we have
\begin{itemize}
\item[\rm (1)] $i_{*}i^{!}(\mathcal{Y})\subseteq \mathcal{Y}$ if and only if $i_{*}i^{*}(\mathcal{X})\subseteq \mathcal{X}$.
\item[\rm (2)] $j_{*}j^{*}(\mathcal{Y})\subseteq \mathcal{Y}$ if and only if $j_{!}j^{*}(\mathcal{X})\subseteq \mathcal{X}$.
\end{itemize}
\end{lemma}
\begin{proof}
For any $X \in \mathcal{X}$, $Y \in \mathcal{Y}$, we have
$$
\begin{aligned}
& \E_\B(X,i_{*}i^{!}(Y))\cong \E_\A(i^{*}(X),i^{!}(Y))\cong  \E_\B(i_{*}i^{*}(X),Y)\\
& \E_\B(X,j_{*}j^{*}(Y))\cong \E_\C(j^{*}(X),j^{*}(Y))\cong  \E_\B(j_{!}j^{*}(X),Y)
\end{aligned}
$$
by Lemma \ref{E-xt}.
It follows that the assertions (1) and (2) are true.
\end{proof}

The following result is stronger than that of \cite{WWZ20R} at some point.
\begin{theorem}\label{main-cotor}
Let $(\A, \B, \C)$ be a recollement of extriangulated categories, and
let $(\X,\Y)$ be a cotorsion pair in $\mathcal{B}$. If $i_{*}i^{*}(\X) \subseteq \X$, $i_{*}i^{*}(\Y) \subseteq \Y$ $j_{*}j^{*}(\Y) \subseteq \Y$, and if $i^{*}, i^!$ are exact, then
\begin{itemize}
\item[\rm (1)] $(i^{*}(\mathcal{X}), i^{!}(\mathcal{Y})=i^*(\Y))$ is a cotorsion pair in $\mathcal{A}$.
\item[\rm (2)] $(j^{*}(\mathcal{X}), j^{*}(\mathcal{Y}))$ is a cotorsion pair in $\C$.
\item[\rm (3)] if $(\X,\Y)$ is hereditary, so are $(i^{*}(\mathcal{X}),i^{!}(\mathcal{Y}))$ and $(j^{*}(\X),j^{*}(\Y))$.
\item[\rm (4)] we have
$$\X=\{B\in \B\mid j^{*}(B)\in j^{*}(\X)\text{ and }i^{*}(B)\in i^{*}(\X)\},$$
$$\Y=\{B\in \B\mid j^{*}(B)\in j^{*}(\Y)\text{ and }i^{!}(B)\in i^{!}(\Y)\}.$$
    \end{itemize}
\end{theorem}

\begin{proof}
(1)
By Lemma \ref{lem-=if}, we know $i_{*}i^{!}(\mathcal{Y})\subseteq \mathcal{Y}$.
By Lemma \ref{lem-=}, we know $i^{*}(\Y)=i^{!}(\Y)$. Now
we show that $(i^{*}(\mathcal{X}), i^{!}(\mathcal{Y}))$ is a cotorsion pair in $\A$.

{\bf (CP1)}
For any $X'\in i^{*}(\X)$, there exists an object $X\in \X$ such that $X'\cong i^{*}(X)$.
Then for any $Y'\in i^{!}(\Y)$,
$$\E_{\A}(X',Y')\cong \E_{\A}(i^{*}(X),Y')\cong \E_{\B}(X,i_{*}(Y')).$$
Notice that $i_{*}(Y')\in i_{*}i^{!}(\Y)\subseteq \Y$, so $\E(X',Y')=0$.

{\bf (CP2)}
For any $A\in \A$, we know that $i_{*}(A)\in \B$.
Then there exist two $\E$-triangles
$$\xymatrix@C=15pt{ Y^1\ar[r]& X^1\ar[r]& i_{*}(A)\ar@{-->}[r]&},$$
$$\xymatrix@C=15pt{ i_{*}(A)\ar[r]& Y_1\ar[r]& X_1\ar@{-->}[r]&}$$
in $\B$ satisfying $X^1, X_1\in\X$, and $Y^1, Y_1\in\Y$. Since $i^{*}$ is exact and $i^*i_*\cong\Id_{\A}$, we deduce the following two $\E$-triangles
$$\xymatrix@C=15pt{i^{*}(Y^1)\ar[r]& i^{*}(X^1)\ar[r]& A\ar@{-->}[r]&},$$
$$\xymatrix@C=15pt{ A\ar[r]& i^{*}(Y_1)\ar[r]& i^{*}(X_1) \ar@{-->}[r]&}$$
in $\A$ satisfying $i^{*}(X^1), i^{*}(X_1)\in i^{*}(\X)$, and $ i^{*}(Y^1),  i^{*}(Y_1)\in i^{*}(\Y)=i^{!}(\Y)$.

Thus $(i^{*}(\X),i^{!}(\Y))$ is a cotorsion pair in $\A$.

(2) By using a similar argument as (1), we can prove that $(j^{*}(\X),j^{*}(\Y))$ is a cotorsion pair in $\C$.

(3) It is clear that
$$\X\subseteq\{B\in \B\mid j^{*}(B)\in j^{*}(\X)\text{ and }i^{*}(B)\in i^{*}(\X)\},$$
$$\Y\subseteq\{B\in \B\mid j^{*}(B)\in j^{*}(\Y)\text{ and }i^{!}(B)\in i^{!}(\Y)\}.$$
Conversely, suppose $X\in \{B\in \B\mid j^{*}(B)\in j^{*}(\X)\text{ and }i^{*}(B)\in i^{*}(\X)\}$.
Since $i^{*}$ is exact, by Lemma \ref{lem-rec}, there exists an $\mathbb{E}$-triangle
$$
  \xymatrix{j_! j^\ast (X)\ar[r]&X\ar[r]&i_\ast i^\ast (X)\ar@{-->}[r]&}
$$
in $\mathcal{B}$.
For any $Y \in \mathcal{Y}$, we obtain the following exact sequence
$$
\xymatrix@C=15pt{
\E_{\B} (i_{*}i^{*}(X), Y) \ar[r]&\E_{\B} (X, Y) \ar[r]&\E_{\B}  (j_{!}j^{*}(X), Y)}.$$
By (1) and (2), we know that  $(i^{*}(\mathcal{X}), i^{!}(\mathcal{Y}))$ and $(j^{*}(\X),j^{*}(\Y))$ are cotorsion pairs in $\A$ and $\C$ respectively, so $\E_{\B}(i_{*}i^{*}(X), Y)\cong \E_{\A}(i^{*}(X), i^{!}(Y))=0$ and $\E_{\B}(j_{!}j^{*}(X), Y)\cong \E_{\C}(j^{*}(X), j^{*}(Y))=0$ by Lemma \ref{E-xt}.
It follows that $\E(X,Y)=0$ and $\E(X,\Y)=0$.
 Then $X\in \X$ by Remark \ref{RemCotors}, and thus
 $$\{B\in \B\mid j^{*}(B)\in j^{*}(\X)\text{ and }i^{*}(B)\in i^{*}(\X)\}\subseteq\X.$$
 Similarly, we can prove
 $$\{B\in \B\mid j^{*}(B)\in j^{*}(\Y)\text{ and }i^{!}(B)\in i^{!}(\Y)\}\subseteq\Y.$$

(4)
{\bf (HCP)}
For any $X'\in i^{*}(\X)$ and any $Y'\in i^{!}(\Y)$, there exist $X\in \X$ and $Y\in \Y$ such that $X'\cong i^{*}(X)$ and $Y\cong i^{!}(Y)$.
Notice that
$$\E^{i}_{\A}(X',Y')\cong \E^{i}_{\A}(i^{*}(X),i^{!}(Y))\cong \E^{i}_{\B}(i_{*}i^{*}(X),Y)$$
for each $i\geq 2$ by Lemma \ref{E-xt}, so $\E^{i}(X',Y')=0$ by the assumption that $i_{*}i^{*}(\X)\subseteq \X$ and $(\X,\Y)$ is hereditary.
Thus $\E^i(i^{*}(\X),i^{!}(\Y))=0$ for each $i\geq 2$, and so $(i^{*}(\X),i^{!}(\Y))$ is hereditary.
Similarly, we can prove that
$(j^{*}(\X),j^{*}(\Y))$ is hereditary.
\end{proof}

Combing Theorem \ref{main-dim} and Theorem \ref{main-cotor}, we have the following the result.
\begin{proposition}
Let $(\mathcal{A},\mathcal{B},\mathcal{C})$ be a recollement of extriangulated categories, and let $(\X,\Y)$ be a hereditary cotorsion pair in $\B$. If $i_{*}i^{*}(\X) \subseteq \X$, $i_{*}i^{*}(\Y) \subseteq \Y$ $j_{*}j^{*}(\Y) \subseteq \Y$, and if $i^{*}, i^!$ are exact, then
$$\max\{i^!(\Y)\text{-}\cores \A, j^*(\Y)\text{-}\cores \C\}\leq\Y\text{-}\cores \B \leq i^!(\Y)\text{-}\cores \A+ j^*(\Y)\text{-}\cores \mathcal{C} + 1.$$
\end{proposition}

By applying Theorem \ref{main-cotor} to triangulated categories, we get the following
\begin{corollary}
{\rm(\cite[Theorem 3.3]{C13C})}
Let $(\A, \B, \C)$ be a recollement of triangulated categories, and
let $(\X,\Y)$ be a cotorsion pair in $\mathcal{B}$. If $i_{*}i^{*}(\X) \subseteq \X$, $i_{*}i^{*}(\Y) \subseteq \Y$ $j_{*}j^{*}(\Y) \subseteq \Y$, then
\begin{itemize}
\item[\rm (1)] $(i^{*}(\mathcal{X}), i^{!}(\mathcal{Y})=i^{*}(\Y))$ is a cotorsion pair in $\mathcal{A}$.

\item[\rm (2)] $(j^{*}(\mathcal{X}), j^{*}(\mathcal{Y}))$ is a cotorsion pair in $\C$.

\item[\rm(3)] we have
$$\X=\{B\in \B\mid j^{*}(B)\in j^{*}(\X)\text{ and }i^{*}(B)\in i^{*}(\X)\},$$
$$\Y=\{B\in \B\mid j^{*}(B)\in j^{*}(\Y)\text{ and }i^{!}(B)\in i^{!}(\Y)\}.$$
    \end{itemize}
\end{corollary}

Conversely, we have the following result.

\begin{proposition}\label{main-conve}
Let $(\A, \B, \C)$ be a recollement of extriangulated categories, and let $(\X,\Y)$ be a pair of subcategories in $\mathcal{B}$.
Assume that $(i^{*}(\mathcal{X}),i^{!}(\mathcal{Y}))$ and $(j^{*}(\X),j^{*}(\Y))$ are cotorsion pairs in $\A$ and $\C$ respectively, and assume that $i_{*}i^{*}(\X) \subseteq \X$, $i_{*}i^{*}(\Y) \subseteq \Y$ $j_{*}j^{*}(\Y) \subseteq \Y$, and $i^{*}, i^{!}$ are exact, then
$(\X,\Y)$ is a cotorsion pair in $\mathcal{C}$ if one of the following conditions is satisfied.
\begin{itemize}
\item[\rm (i)] $i_{*}i^{*}\cong \Id_{\B}$.
\item[\rm (ii)] $j_{!}j^{*}\cong {\Id}_{\B}$ and $j_{!}j^{*}(\Y)\subseteq\Y$.
\end{itemize}
Moreover, if $(i^{*}(\mathcal{X}),i^{!}(\mathcal{Y}))$ or $(j^{*}(\X),j^{*}(\Y))$ is hereditary, then so is $(\X,\Y)$.
\end{proposition}
\begin{proof}
{\bf (CP1)}
For any $X\in \X$ and $Y\in\Y$.
Since $i^{*}$ is exact by assumption, by Lemma \ref{lem-rec}, there exists an $\mathbb{E}$-triangle
$$
  \xymatrix{j_! j^\ast (X)\ar[r]&X\ar[r]&i_\ast i^\ast (X)\ar@{-->}[r]&}
$$
in $\mathcal{B}$.
Then we obtain the following exact sequence
$$
\xymatrix@C=15pt{
\E_{\B} (i_{*}i^{*}(X), Y) \ar[r]&\E_{\B} (X, Y) \ar[r]&\E_{\B}  (j_{!}j^{*}(X), Y)}.$$
 Notice that $(i^{*}(\mathcal{X}), i^{!}(\mathcal{Y}))$ and $(j^{*}(\X),j^{*}(\Y))$ are cotorsion pairs in $\A$ and $\C$ respectively, so
 $$\E_{\B} (i_{*}i^{*}(X), Y)\cong \E_{\A} (i^{*}(X), i^{!}(Y))=0~~\mbox{and}~~\E_{\B} (j_{!}j^{*}(X), Y)\cong \E_{\C} (j^{*}(X), j^{*}(Y))=0$$
  by Lemma \ref{E-xt}.
It follows that $\E(X,Y)=0$ and $\E(\X,\Y)=0$.

{\bf (CP2)}
{(i)}
For any $B\in \B$, we know $i^{*}(B)\in \A$.
Then there exist the following two $\E$-triangles
$$
\xymatrix@C=15pt{
i^{*}(Y^1) \ar[r]& i^{*}(X^1) \ar[r] &i^{*}(B) \ar@{-->}[r]&},
$$
$$
\xymatrix@C=15pt{
i^{*}(B) \ar[r]& i^{*}(Y_1) \ar[r]& i^{*}(X_1) \ar@{-->}[r]&}
$$
in $\A$ with $i^{*}(X^1), i^{*}(X_1)\in i^{*}(\X)$ and $i^{*}(Y^1), i^{*}(Y_1)\in i^{*}(\Y)=i^!(\Y)$. Since $i_{*}$ is exact and $i_{*} i^{*} \cong \operatorname{Id}_{\B}$, we get the following $\E$-triangles
$$
\xymatrix@C=15pt{
i_{*}i^{*}(Y^1) \ar[r]& i_{*}i^{*}(X^1) \ar[r] &B\ar@{-->}[r]&},
$$
$$
\xymatrix@C=15pt{
B\ar[r]&i_{*} i^{*}(Y_1) \ar[r]&i_{*} i^{*}(X_1) \ar@{-->}[r]&}
$$
in $\B$ with $i_{*}i^{*}(X^1), i_{*}i^{*}(X_1)\in i_{*}i^{*}(\X)\subseteq \X$ and $i_{*}i^{*}(Y^1), i_{*}i^{*}(Y_1)\in i_{*}i^{*}(\Y)\subseteq \Y$.

{(ii)}
For any $B\in \B$, we know $j^{*}(B)\in \C$.
Then there exist the following two $\E$-triangles
$$
\xymatrix@C=15pt{
j^{*}(Y^1) \ar[r]& j^{*}(X^1) \ar[r] &j^{*}(B) \ar@{-->}[r]&},
$$
$$
\xymatrix@C=15pt{
j^{*}(B) \ar[r]& j^{*}(Y_1) \ar[r]&j^{*}(X_1) \ar@{-->}[r]&}
$$
in $\C$ with $j^{*}(X^1), j^{*}(X_1)\in j^{*}(\X)$ and $j^{*}(Y^1), j^{*}(Y_1)\in j^{*}(\Y)$. Since $j_{!}$ is exact by Lemma \ref{lem-rec} and $j_{!}j^{*} \cong {\Id}_{\B}$ by assumption, we get the following $\E$-triangles
$$
\xymatrix@C=15pt{
j_{!}j^{*}(Y^1) \ar[r]&j_{!} j^{*}(X^1) \ar[r] &B \ar@{-->}[r]&},
$$
$$
\xymatrix@C=15pt{
B\ar[r]& j_{!}j^{*}(Y_1) \ar[r]&j_{!}j^{*}(X_1) \ar@{-->}[r]&}
$$
in $\B$ with $j_{!}j^{*}(X^1), j_{!}j^{*}(X_1)\in j_{!}j^{*}(\X)\subseteq \X$ and $j_{!}j^{*}(Y^1), j_{!}j^{*}(Y_1)\in j_{!}j^{*}(\Y)\subseteq \Y$, where the fact $ j_{!}j^{*}(\X)\subseteq \X$ follows from Lemma \ref{lem-=if} and the assumption $ j_{*}j^{*}(\Y)\subseteq \Y$.

Thus $(\X,\Y)$ is a cotorsion pair in $\B$.

Moreover, for any $X\in \X$ and $Y\in \Y$.
If $(i^{*}(\mathcal{X}),i^{!}(\mathcal{Y}))$ is hereditary, then
$$
\E_{\B}^i(X, Y) \cong \E_{\B}^i(X, i_{*}i^{*}(Y)) \cong \E_{\A}^i(i^{*}(X), i^{*}(Y))=0.
$$
for any $i\geq 2$.

If $(j^{*}(\X),j^{*}(\Y))$ is hereditary, then
$$
\E_{\B}^i(X, Y) \cong \E_{\B}^i(j_{!}j^{*}(X), Y) \cong \E_{\A}^i(j^{*}(X), j^{*}(Y))=0.
$$
for any $i\geq 2$.

Then $(\X,\Y)$ is hereditary.
\end{proof}

The following result discusses the cotorsion triples in recollements of extriangulated categories.
\begin{theorem}\label{prop-cot-trip}
Let $(\A, \B, \C)$ be a recollement of extriangulated categories, and let $\mathcal{X}$, $\mathcal{Y}$, $\mathcal{Z}$ be subcategories of $\B$.
Assume that $i_{*}i^{*}(\X) \subseteq \X$, $i_{*}i^{*}(\Y) \subseteq \Y, i_{*}i^{*}(\mathcal{Z}) \subseteq \mathcal{Z}, j_{*}j^{*}(\mathcal{Y}) \subseteq \mathcal{Y}, j_{*}j^{*}(\mathcal{Z}) \subseteq \mathcal{Z}$ and $i^{*}, i^!$ are exact, then the following statements hold.
\begin{itemize}
\item[\rm (1)] If $(\mathcal{X}, \mathcal{Y}, \mathcal{Z})$ is a cotorsion triple in $B$, then
\begin{itemize}
\item[\rm (a)] $(i^{*}(\mathcal{X}), i^{*}(\mathcal{Y})=$ $i^{!}(\mathcal{Y}), i^{!}(\mathcal{Z}))$ is a cotorsion triple in $\A$.
\item[\rm (b)] $(j^{*}(\mathcal{X}), j^{*}(\mathcal{Y}), j^{*}(\mathcal{Z}))$ is a cotorsion triple in $\C$.
\item[\rm (c)] if $(\mathcal{X}, \mathcal{Y}, \mathcal{Z})$ is hereditary, so are $(i^{*}(\mathcal{X}), i^{*}(\mathcal{Y})=i^{!}(\mathcal{Y}), i^{!}(\mathcal{Z}))$ and $(j^{*}(\mathcal{X}), j^{*}(\mathcal{Y}),j^{*}(\mathcal{Z}))$.
\end{itemize}
\item[\rm (2)] If $(i^{*}(\mathcal{X}), i^{*}(\mathcal{Y})=$ $i^{!}(\mathcal{Y}), i^{!}(\mathcal{Z}))$ and $(j^{*}(\mathcal{X}), j^{*}(\mathcal{Y}), j^{*}(\mathcal{Z}))$ are cotorsion triples in $\A$ and $\C$ respectively, then $(\X,\Y,\Z)$ is a cotorsion triple in $\mathcal{B}$ if one of the following conditions is satisfied.
\begin{itemize}
\item[\rm (i)] $i_{*}i^{*}\cong {\Id}_{\B}$.  %, and $i_{*}i^{!}\cong {\Id}_{\B}$;
\item[\rm (ii)] $j_{!}j^{*}\cong {\Id}_{\B}$ and $j_{*}j^{*}\cong {\Id}_{\B}$;
\end{itemize}
Moreover, if $(i^{*}(\mathcal{X}), i^{*}(\mathcal{Y})=$ $i^{!}(\mathcal{Y}), i^{!}(\mathcal{Z}))$ or $(j^{*}(\mathcal{X}), j^{*}(\mathcal{Y}), j^{*}(\mathcal{Z}))$ is hereditary, then so is $(\X,\Y,\Z)$.
    \end{itemize}
\end{theorem}
\begin{proof}
(1)
(a) By Theorem \ref{main-cotor}, we have that $(i^{*}(\X),i^{!}(\Y)=i^{*}(\Y))$ and $(i^{*}(\Y),i^{!}(\Z))$ are cotorsion pairs in $\A$. Then
$(i^{*}(\X), i^{*}(\Y)=i^{!}(\Y), i^{!}(\Z))$ is a cotorsion triple in $\A$.

(b)
By Theorem \ref{main-cotor}, we know that $(j^{*}(\X),j^{*}(\Y))$ and $(j^{*}(\Y),j^{*}(\Z))$ are cotorsion pairs in $\C$.
So $(j^{*}(\X),j^{*}(\Y), j^{*}(\Z))$ is a cotorsion triple in $\C$.

(c)
Using a similar argument as Theorem \ref{main-cotor}(3), we obtain the assertion.

(2)
(a) It follows from Proposition \ref{main-conve} that $(\X,\Y)$ is a cotorsion pair in $\B$.
Using a similar argument as Proposition \ref{main-conve}~{\bf (CP1)}, we get $\E(\Y,\Z)=0$.

Now we show that $(\Y,\Z)$ is a cotorsion pair in $\B$.

{\bf (CP2)}
(i) For any $B\in \B$, we know $i^{*}(B)\in \A$.
Then there exist the following two $\E$-triangles
$$
\xymatrix@C=15pt{
i^{*}(Z^1) \ar[r]& i^{*}(Y^1) \ar[r] &i^{*}(B) \ar@{-->}[r]&},
$$
$$
\xymatrix@C=15pt{
i^{*}(B) \ar[r]& i^{*}(Z_1) \ar[r]& i^{*}(Y_1) \ar@{-->}[r]&}
$$
in $\A$ with $i^{*}(Y^1), i^{*}(Y_1)\in i^{*}(\Y)$ and $i^{*}(Z^1), i^{*}(Z_1)\in i^{*}(\Z)=i^{!}(\Z)$. Since $i_{*}$ is exact and $i_{*} i^{*} \cong {\Id}_{\B}$, we get the following $\E$-triangles
$$
\xymatrix@C=15pt{
i_{*}i^{*}(Z^1) \ar[r]& i_{*}i^{*}(Y^1) \ar[r] &B \ar@{-->}[r]&},
$$
$$
\xymatrix@C=15pt{
B\ar[r]& i_{*}i^{*}(Z_1) \ar[r]&i_{*} i^{*}(Y_1) \ar@{-->}[r]&}
$$
in $\B$ with $i_{*}i^{*}(Y^1), i_{*}i^{*}(Y_1)\in i_{*}i^{*}(\Y)\subseteq \Y$ and $i_{*}i^{*}(Z^1), i_{*}i^{*}(Z_1)\in i_{*}i^{*}(\Z)\subseteq \Z$.

{(ii)}
For any $B\in \B$, we know $j^{*}(B)\in \C$.
Then there exist the following two $\E$-triangles
$$
\xymatrix@C=15pt{
j^{*}(Z^1) \ar[r]& j^{*}(Y^1) \ar[r] &j^{*}(B) \ar@{-->}[r]&},
$$
$$
\xymatrix@C=15pt{
j^{*}(B) \ar[r]& j^{*}(Z_1) \ar[r]&j^{*}(Y_1) \ar@{-->}[r]&}
$$
in $\C$ with $j^{*}(Y^1), j^{*}(Y_1)\in j^{*}(\Y)$ and $j^{*}(Z^1), j^{*}(Z_1)\in j^{*}(\Z)$. Since $j_{*}$ is exact by Lemma \ref{lem-rec} and $j_{*}j^{*} \cong {\Id}_{\B}$, we get the following $\E$-triangles
$$
\xymatrix@C=15pt{
j_{*}j^{*}(Z^1) \ar[r]&j_{*} j^{*}(Y^1) \ar[r] &B \ar@{-->}[r]&},
$$
$$
\xymatrix@C=15pt{
B \ar[r]& j_{*}j^{*}(Z_1) \ar[r]&j_{*}j^{*}(Y_1) \ar@{-->}[r]&}
$$
in $\B$ with $j_{*}j^{*}(Y^1), j_{*}j^{*}(Y_1)\in j_{*}j^{*}(\Y)\subseteq \Y$ and $j_{*}j^{*}(Z^1), j_{*}j^{*}(Z_1)\in j_{*}j^{*}(\Z)\subseteq \Z$.

Then $(\Y,\Z)$ is a cotorsion pair in $\B$.

Moreover, we can obtain the final assertion by using a similar argument as Proposition \ref{main-conve}.
\end{proof}

By applying Proposition \ref{prop-cot-trip} to abelian categories, we get the following result.

\begin{corollary}{\rm (cf. \cite[Theorem 2.11]{FH22T})}
Let $(\A, \B, \C)$ be a recollement of abelian categories, and let $\mathcal{X}$, $\mathcal{Y}$, $\mathcal{Z}$ be subcategories of $\B$.
Assume that $i_{*}i^{*}(\X) \subseteq \X$, $i_{*}i^{*}(\Y) \subseteq \Y, i_{*}i^{*}(\mathcal{Z}) \subseteq \mathcal{Z}, j_{*}j^{*}(\mathcal{Y}) \subseteq \mathcal{Y}, j_{*}j^{*}(\mathcal{Z}) \subseteq \mathcal{Z}$ and $i^{*}, i^{!}$ are exact, then the following statements hold.
\begin{itemize}
\item[\rm (1)] If $(\mathcal{X}, \mathcal{Y}, \mathcal{Z})$ is a hereditary cotorsion triple in $\B$, then
\begin{itemize}
\item[\rm (a)] $(i^{*}(\mathcal{X}), i^{*}(\mathcal{Y})=$ $i^{!}(\mathcal{Y}), i^{!}(\mathcal{Z}))$ is a hereditary cotorsion triple in $\A$.
\item[\rm (b)] $(j^{*}(\mathcal{X}), j^{*}(\mathcal{Y}), j^{*}(\mathcal{Z}))$ is a hereditary cotorsion triple in $\C$.
\end{itemize}
\item[\rm (2)] If $(i^{*}(\mathcal{X}), i^{*}(\mathcal{Y})=$ $i^{!}(\mathcal{Y}), i^{!}(\mathcal{Z}))$ and $(j^{*}(\mathcal{X}), j^{*}(\mathcal{Y}), j^{*}(\mathcal{Z}))$ are hereditary cotorsion triples in $\A$ and $\C$ respectively, then
$(\X,\Y,\Z)$ is a hereditary cotorsion triple in $\mathcal{B}$ if one of the following conditions is satisfied.
\begin{itemize}
\item[\rm (i)] $i_{*}i^{*}\cong {\Id}_{\B}$.
\item[\rm (ii)] $j_{!}j^{*}\cong {\Id}_{\B}$ and $j_{*}j^{*}\cong {\Id}_{\B}$.
\end{itemize}
    \end{itemize}
\end{corollary}

\section{Construct a recollement of abelian categories}

Following \cite{Liu}, let $(\X,\Y)$ be a cotorsion pair in $\mathcal{C}$. $\C^+={\rm cone}(\Y, \X\cap\Y)$ and $\C^-={\rm cocone}(\X\cap\Y,\X)$.
It is clear that each of $\C^{+}$and $\C^{-}$ contains $\X \cap \Y$.
Denote the quotient of $\C^{+} \cap \C^{-}$by $\X \cap \Y$ as
$$
\underline{\mathcal{H}}=\C^{+} \cap \C^{-} / \X \cap \Y .
$$
Namely, $\underline{\mathcal{H}}$ has the same objects as $\C^{+} \cap \C^{-}$.
In particular, if $\X$ is a cluster tilting subcategory of $\C$, then $\C^+=\C=\C^-$.

For any $C, D \in \C^{+} \cap \C^{-}$,
$${\Hom}_{\underline{\mathcal{H}}}(C, D)={\Hom}_{\C^{+} \cap \C^{-}}(C, D) /\{f \in {\Hom}_{\C^{+} \cap \C^{-}}(C, D) \mid f \text{ factors through some object in } X\cap \Y\}.$$
 For any $f \in {\Hom}_{\C^{+} \cap \C^{-}}(C, D)$, denote its image in
 ${\Hom}_{\underline{\mathcal{H}}}(C, D)$ by $\underline{f}$. Then $\underline{\mathcal{H}}$ is an additive category which is called the heart of $(\X, \Y)$. Furthermore, we have the following theorem.

\begin{theorem}{\rm (\cite[Theorem 3.2]{Liu})}
For any cotorsion pair $(\X, \Y)$ in $\C$, the heart $\underline{\mathcal{H}}$ is an abelian category.
\end{theorem}

\begin{proposition}\label{adjoint-}
{\rm (\cite[Proposition 4.3]{C13C})}
Suppose that $\A$ and $\B$ are two categories, and $(F, G)$ is an adjoint pair of functor with $F: \A \rightarrow \B$ and $G: \B\rightarrow \A$. Let $\X$ and $\Y$ be the subcategories of $\A$ and $\B$, respectively, such that $F(\X) \subseteq \Y$ and $G(\Y) \subseteq \X$. Then the following statements hold:
\begin{itemize}
\item[\rm (1)] The pair $(F, G)$ can be canonically extended to an adjoint pair $(\underline{F}, \underline{G})$ of functors between the quotient categories $\A / \X$ and $\B / \Y$ with $\underline{F}: \A / \X \rightarrow \B / \Y$ and $\underline{G}$ : $B / \Y \rightarrow \A / \X$
\item[\rm (2)] If $F$ (respectively, $G$) is fully faithful, then so is $\underline{F}$ (respectively, $\underline{G}$.
\end{itemize}
\end{proposition}

We are now in a position to provide a method for constructing a recollement of abelian categories.

\begin{theorem}\label{main-construc}
Let $(\A,\B,\C)$ be a recollement of extriangulated categories, and let $(\X, \Y)$ be a cotorsion pair in $\B$.
If $i_{*}i^{!}(\X) \subseteq \X$, $i_{*}i^{*}(\X) \subseteq \X$, $i_{*}i^{*}(\Y) \subseteq \Y$, $j_{*}j^{*}(\X) \subseteq \X$, $j_{*}j^{*}(\Y) \subseteq \Y$, $j_{!}j^{*}(\Y) \subseteq \Y$ and if $i^{*}$, $i^{!}$ are exact, then the abelian category $\underline{\mathcal{H}}$ admits a recollement relative to abelian categories $\underline{\mathcal{H'}}$ and $\underline{\mathcal{H''}}$ as follows
$$
  \xymatrix{\underline{\mathcal{H'}}\ar[rr]|{\underline{i_{*}}}&&\ar@/_1pc/[ll]|{\underline{i^{*}}}\ar@/^1pc/[ll]|{\underline{i^{!}}}\underline{\mathcal{H}}
\ar[rr]|{\underline{j^{\ast}}}&&\ar@/_1pc/[ll]|{\underline{j_{!}}}\ar@/^1pc/[ll]|{\underline{j_{\ast}}}\underline{\mathcal{H''}}}
$$
where $\underline{\mathcal{H}}, \underline{\mathcal{H'}}, \underline{\mathcal{H''}}$ are the hearts of $(\X, \Y)$, $(i^*(\X),i^{!}(\Y))$ and $(j^*(\X), j^*(\Y))$, respectively.
\end{theorem}
\begin{proof}
By Theorem \ref{main-cotor}, we know that $(i^{*}(\X),i^{!}(\Y))$ and $(j^{*}(\X),j^{*}(\Y))$ are cotorsion pairs in $\A$ and $\C$ respectively.
We first claim that $j^*(\B^{+})\subseteq \C^{+}$ and $j_*(\C^{+})\subseteq \B^{+}$.
In fact, for any $B\in \B^{+}$, there exists an $\E$-triangle
$$\xymatrix@C15pt{Y\ar[r]&W \ar[r]&B\ar@{-->}[r]&}$$
in $\B$ with $W \in \X\cap \Y$ and $Y \in \Y$.
One can get an $\E$-triangle
$$\xymatrix@C15pt{j^{*}(Y)\ar[r]&j^{*}(W) \ar[r]&j^{*}(B)\ar@{-->}[r]&}$$
in $\B$.
It is clear that $j^*(W) \in$ $j^*(\X \cap \Y) \subseteq j^*(\X) \cap j^*(\Y)$ and $j^*(Y) \in j^*(\Y)$.
Hence $j^*(B) \in \C^{+}$ and $j^{*}(\B^+)\subseteq \C^{+}$.

On the other hand, for any $C \in \C^{+}$, there exists an $\E$-triangle
$$\xymatrix@C15pt{Y''\ar[r]&W'' \ar[r]&C\ar@{-->}[r]&}$$
in $\C$ with $W'' \in j^{*}(\X)\cap j^{*}(\Y)$ and $Y'' \in j^{*}(\Y)$.
Since $j_{*}$ is exact, one can get that
$$\xymatrix@C15pt{j_{*}(Y'')\ar[r]&j_{*}(W'') \ar[r]&j_{*}(C)\ar@{-->}[r]&}$$
is an $\E$-triangle in $\B$, where $j_*(W'') \in j_*( j^{*}(\X)\cap j^{*}(\Y)) \subseteq j_*j^{*}(\X)\cap j_*j^{*}(\Y)\subseteq \X \cap \Y$
and $j_*(Y'') \in j_*(j^*(\Y)) \subseteq \Y$.
Hence $j_*(C) \in\B^{+}$ and $j_*(\C^+) \in\B^{+}$.

Simillary, we can prove the assertions $j^*(\B^{-})\subseteq \C^{-}$ and $j_*(\C^{-})\subseteq \B^{-}$.

Then, we have
$j^*(\B^{+} \cap \B^{-}) \subseteq j^*(\B^{+})\cap j^{*}(\B^{-})\subseteq\C^{+} \cap \C^{-}$ and $j_*(\C^{+} \cap \C^{-}) \subseteq j_*(\C^{+}) \cap j_{*}(\C^{-}) \subseteq \B^{+} \cap \B^{-}$.
By Proposition \ref{adjoint-},
the adjoint pair $(j^{*},j_{*})$ induces an adjoint pair $(\underline{j^*}, \underline{j_*})$ with $\underline{j^*}: \underline{\mathcal{H}} \rightarrow \underline{\mathcal{H''}}$ and $\underline{j_*}: \underline{\mathcal{H''}} \rightarrow \underline{\mathcal{H}}$, such that $\underline{j_*}$ is fully faithful.

Using similar arguments as above, we can prove that $j_!: \C \rightarrow \B$ induces an additive functor of abelian categories $\underline{j_!}: \underline{\mathcal{H''}} \rightarrow \underline{\mathcal{H}}$ such that $(\underline{j_!}, \underline{j^*})$ is a adjoint pair and $\underline{j_{!}}$ is a fully faithful;
$i^*, i_*, i^!$ induce additive functors of abelian categories $\underline{i^*}$ : $\underline{\mathcal{H}} \rightarrow \underline{\mathcal{H'}}$, $\underline{i_*}: \underline{\mathcal{H'}} \rightarrow \underline{\mathcal{H}}$, $\underline{\underline{i}}^{!}: \underline{\mathcal{H}} \rightarrow \underline{\mathcal{H'}}$, respectively, such that $(\underline{i^*}, \underline{i_*})$, $(\underline{i_*}, \underline{i^{!}})$ are adjoint pairs and $\underline{i_*}$ is a fully faithful.
Since $j^* i_*=0$, we get $\underline{j}^* \underline{i}_*=0$. Then we finish the proof.
\end{proof}
In particular, for a cluster tilting subcategory $\X$ of $\B$, %that is, $(\X,\X)$ is a cotorsion pair in $\B$,
by Lemma \ref{lem-=if}, the assertions $i_{*}i^{*}(\X) \subseteq \X$ and $j_{*}j^{*}(\Y) \subseteq \Y$ can deduce all ``$\subseteq$'' in Theorem \ref{main-construc}.
Then we have the following result.

\begin{corollary}{\rm (\cite[Theorem 4.5]{HHZ21T})}
Let $(\A,\B,\C)$ be a recollement of extriangulated categories, and let $\X$ be a cluster tilting subcategory of $\B$.
If $i_{*}i^{*}(\X) \subseteq \X$, $j_{*}j^{*}(\Y) \subseteq \Y$, and if $i^{*}$, $i^{!}$ are exact, then the abelian category $\B/\X$ admits a recollement relative to abelian categories $\A/i^*(\X)$ and $\C/j^{*}(\X)$ as follows
$$
  \xymatrix{\A/i^*(\X)\ar[rr]|{\underline{i_{*}}}&&\ar@/_1pc/[ll]|{\underline{i^{*}}}\ar@/^1pc/[ll]|{\underline{i^{!}}}\B/\X
\ar[rr]|{\underline{j^{\ast}}}&&\ar@/_1pc/[ll]|{\underline{j_{!}}}\ar@/^1pc/[ll]|{\underline{j_{\ast}}}
\C/j^*(\X)}.
$$
\end{corollary}

By applying Theorem \ref{main-construc} to triangulated categories, we get the following the result.
\begin{corollary}
{\rm (\cite[Theorem 4.4]{C13C})}
Let $(\A,\B,\C)$ be a recollement of triangulated categories. and let $(\X, \Y)$ be a cotorsion pair of $\B$.
If $i_{*}i^{!}(\X) \subseteq \X$, $i_{*}i^{*}(\X) \subseteq \X$, $i_{*}i^{*}(\Y) \subseteq \Y$, $j_{*}j^{*}(\X) \subseteq \X$, $j_{*}j^{*}(\Y) \subseteq \Y$, $j_{!}j^{*}(\Y) \subseteq \Y$, then the abelian category $\underline{\mathcal{H}}$ admits a recollement relative to abelian categories $\underline{\mathcal{H'}}$ and $\underline{\mathcal{H''}}$ as follows
$$
  \xymatrix{\underline{\mathcal{H'}}\ar[rr]|{\underline{i_{*}}}&&\ar@/_1pc/[ll]|{\underline{i^{*}}}\ar@/^1pc/[ll]|{\underline{i^{!}}}\underline{\mathcal{H}}
\ar[rr]|{\underline{j^{\ast}}}&&\ar@/_1pc/[ll]|{\underline{j_{!}}}\ar@/^1pc/[ll]|{\underline{j_{\ast}}}\underline{\mathcal{H''}}}
$$
where $\underline{\mathcal{H}}, \underline{\mathcal{H'}}, \underline{\mathcal{H''}}$ are the hearts of $(\X, \Y)$, $(i^*(\X),i^{!}(\Y))$ and $(j^*(\X), j^*(\Y))$, respectively.
\end{corollary}

{\bf Xin Ma}\\
College of Science, Henan University of Engineering, 451191 Zhengzhou, Henan, P. R. China\\
E-mail: maxin@haue.edu.cn
\\[0.3cm]
\textbf{Panyue Zhou}\\
School of Mathematics and Statistics, Changsha University of Science and Technology, 410114 Changsha, Hunan,  P. R. China\\
E-mail: panyuezhou@163.com


\begin{thebibliography}{99}
\setlength{\itemsep}{-5pt}


\bibitem{AT22H}
T. Adachi, M. Tsukamoto,
{\it Hereditary cotorsion pairs and silting subcategories in extriangulated categories}, J. Algebra {\bf 594} (2022), 109--137.

\bibitem{BBD}
A. Be{\u\i}linson, J. Bernstein, P. Deligne,
 Faisceaux pervers,
 in: Analysis and topology on singular spaces, {I}, Luminy,
  1981, {Ast\'erisque}, {vol. 100}, Soc. Math. France,
  Paris, 1982, 5--171.




\bibitem{Ben} R. Bennett-Tennenhaus, A. Shah, {\it Transport of structure in higher homological algebra}, J. Algebra
{\bf 574} (2021), 514--549.

% \bibitem{Bu} T. B\"{u}hler, {\it Exact categories}, Expo. Math. {\bf 28} (2010), 1--69.

\bibitem{C13C} J. Chen, {\it Cotorsion pairs in a recollement of triangulated categories},
Comm. Algebra {\bf 41} (2013), 2903--2915.

\bibitem{Fr} V. Franjou, T. Pirashvili, {\it Comparison of abelian categories recollements}, Doc. Math. {\bf 9} (2004), 41--56.
    \bibitem{FH22T} X. Fu, Y. Hu, {\it The recollements of abelian categories: cotorsion dimensions and cotorsion triples},
Bull. Iranian Math. Soc. {\bf 48} (2022), 963--977.
\bibitem{GNP21} M. Gorsky, H. Nakaoka, Y. Palu, {\it Positive and negative extensions in extriangulated categories}, arXiv:2103.12482.
\bibitem{GMT} W. Gu, X. Ma, L. Tan, {\it Homological dimensions of extriangulated categories and recollements}, 	 arXiv:2104.06042.


\bibitem{HHZ21T} J. He, Y. Hu, P. Zhou, {\it Torsion pairs and recollements of extriangulated categories}, Comm. Algebra {\bf 50} (2022), 2018--2036.




\bibitem{HZZ21G} J. Hu, D. Zhang, P. Zhou. {\it Gorenstein homological dimensions for extriangulated categories}.
Bull. Malays. Math. Sci. Soc. {\bf 44}  (2021) 2235--2252.

\bibitem{HZZ20P} J. Hu, D. Zhang, P. Zhou, {\it Proper classes and Gorensteinness in extriangulated categories},
J. Algebra {\bf 551}(2020), 23--60.

\bibitem{HZZ21P} J. Hu, D. Zhang, P. Zhou, {\it Proper resolutions and Gorensteinness in extriangulated categories}, Front. Math. China {\bf 16} (2021), 95--117.



\bibitem{HZ21R} Y. Hu, P. Zhou, {\it Recollements arising from cotorsion pairs on extriangulated categories}, Front. Math. China {\bf 16} (2021), 937--955.



 \bibitem{INY18A} O. Iyama, H. Nakaoka, Y. Palu, {\it Auslander-Reiten theory in extriangulated categories}, arXiv: 1805.03776.


\bibitem{K} C. Klapproth, $n$-extension closed subcategories of $n$-exangulated categories. arXiv: 2209.01128v3.


\bibitem{L13H} Y. Liu, {\it Hearts of twin cotorsion pairs on exact categories}, J. Algebra {\bf 394} (2013), 245--284.

\bibitem{Liu} Y. Liu, H. Nakaoka, {\it Hearts of twin cotorsion pairs on extriangulated categories},
J. Algebra {\bf 528} (2019), 96--149.



% \bibitem{L17G} M. Lu, {\it Gorenstein defect categories of triangular matrix algebras}, J. Algebra {\bf 480} (2017), 346--367.



\bibitem{MZZ} X. Ma, T. Zhao, X. Zhuang, {\it Resolving subcategories and dimensions in recollements of extriangulated categories}, Bull. Malays. Math. Sci. Soc. {\bf 46} (2023), 36.



\bibitem{MD06} L. Mao, N. Ding, {\it The cotorsion dimension of modules and rings}, Lect. Notes Pure Appl. Math. {\bf 249} (2006), 217--233.

\bibitem{Na} H. Nakaoka, Y. Palu, {\it Extriangulated categories, Hovey twin cotorsion pairs and model structures}, Cah. Topol. G\'{e}om. Diff\'{e}r. Cat\'{e}g. {\bf 60} (2019), 117--193.

\bibitem{N13G} H. Nakaoka, {\it General heart construction on a triangulated category (I): Unifying $t$-structures and cluster tilting subcategories}, Appl. Categ. Structures {\bf 19} (2011), 879--899.

\bibitem{PC14H} C. Psaroudakis, {\it Homological theory of recollements of abelian categories}, J. Algebra {\bf 398} (2014), 63--110.

\bibitem{WWZ20R} L. Wang, J. Wei, H. Zhang, {\it Recollements of extriangulated categories}, Colloq. Math. {\bf 167} (2022), 239--259.

\bibitem{ZZ18T} P. Zhou, B. Zhu,
{\it Triangulated quotient categories revisited},
J. Algebra {\bf 502} (2018), 196--232.

\bibitem{ZZ19C} P. Zhou, B. Zhu,
{\it Cluster-tilting subcategories in extriangulated categories},
Theory Appl. Categ. {\bf 34} (2019), 221--242.


\bibitem{ZZ20T} B. Zhu, X. Zhuang, {\it Tilting subcategories in extriangulated categories}, Front. Math. China {\bf 15} (2020), 225--253.
\bibitem{ZZ21G} B. Zhu, X. Zhuang, {\it Grothendieck groups in extriangulated categories}, J. Algebra {\bf 574} (2021), 206--232.


\end{thebibliography}
\end{document}